\documentclass[12pt]{article}

\usepackage{multicol}

\usepackage{amsthm,amssymb,amsmath}

\newtheorem{theorem}{Theorem}
\newtheorem{lemma}[theorem]{Lemma}

\newtheorem{example}[theorem]{Example}

\newcommand{\GF}{\textnormal{GF}}

\title{On completing three cyclic transversals to a latin square}
\author{Nicholas J. Cavenagh\\
School of Mathematics and Statistics\\
The University of New South Wales\\
Sydney 2052, Australia\\
\\
Carlo H\"{a}m\"{a}l\"{a}inen \\
Department of Mathematics \\
The University of Queensland \\
QLD 4072, Australia \\
\\
Adrian M. Nelson\\
School of Mathematics and Statistics\\
University of Sydney\\
NSW 2006, Australia\\
}

\begin{document}
\maketitle

\begin{abstract}
Let $P$ be a partial latin square of prime order $p>7$ consisting of
three cyclically generated transversals. Specifically, let $P$ be a
partial latin square of the form:
\[
P=\{(i,c+i,s+i),(i,c'+i,s'+i),(i,c''+i,s''+i)\mid 0 \leq i< p\}
\]
for some distinct $c,c',c''$ and some distinct $s,s',s''$. In this
paper we show that any such $P$ completes to a latin square which is
diagonally cyclic.
\end{abstract}


\section{Background information}\label{sec1}
\setcounter{theorem}{0}

A {\em latin square} of order $n$ is an $n\times n$ array of symbols such that
 each cell contains one symbol and each symbol occurs once in each row and once in each column.
In this paper, rows, columns and symbols are taken from the set
$N=\{0,1,\dots ,n-1\}$ and are always calculated modulo $n$. We use
the notation $i\circ j$ to denote the symbol in cell $(i,j)$ of a
latin square $L=L^{\circ}$. We may denote $L$ a set of ordered (row,
column, symbol) triples of the form $(i,j,i\circ j)$. A {\em partial
latin square}, as its name suggests, is a partially filled-in
$n\times n$ array of symbols such that each cell contains at {\em
most} one symbol and each symbol occurs at most once in each row and
at at most once in each column.

A latin square $L^{\circ}$ is said to be {\em diagonally cyclic} if
$n$ is odd and for each cell $(i,j)$, $i\circ j= k$ implies that
$(i+1)\circ (j+1)= k+1$. Diagonally cyclic squares of even order do
not exist; for a nice proof see~\cite{gr2}. Henceforth we assume
that $n$ is odd.

Let $j$ be coprime to some odd $n$ and $j\neq 1$. If we define
$\circ_j$ by $0\circ_j i=ij$ (mod $p$) for each $i$, a valid
diagonally cyclic latin square is generated, which we denote by
$B_{n,j}$. The following diagram shows $B_{5,3}$.

\[
\begin{array}{|c|c|c|c|c|}
\hline
0 & 3 & 1 & 4 & 2 \\
\hline
3 & 1 & 4 & 2 & 0 \\
\hline
1 & 4 & 2 & 0 & 3 \\
\hline
4 & 2 & 0 & 3 & 1 \\
\hline
2 & 0 & 3 & 1 & 4  \\
\hline
\end{array}
\]

Clearly a diagonally cyclic latin square is defined by the ordering
of symbols in row $0$. However, not all orderings of the first row
will complete to a diagonally cyclic latin square, as symbols may
not be repeated within a column. In fact, row $0$ ``generates'' a
diagonally cyclic latin square if and only if the operation
$f(i)=0\circ i - i$ is a permutation of the set $N$.

Let $L$ be a diagonally cyclic latin square of order $n$. For any
constant integer $c$, we define $L\oplus_1 c$ to be the diagonally
cyclic latin square formed by cycling the rows of $L$ by $c$ (modulo
$n$). That is,
\[
L\oplus_1 c = \{(i+c,j,k)\mid (i,j,k)\in L\}.
\]
Clearly such a transformation is invertible, and thus gives rise to
equivalence classes of diagonally cyclic latin squares. We thus
often assume, without loss of generality, that $0\circ 0=0$. In
fact, if $0\circ 0=0$ we say that the diagonally cyclic latin square
is in {\em standard order}. Similarly we define $L\oplus_2 c$ and
$L\oplus_3 c$ to be the diagonally latin squares formed by adding
constant $c$ to each column or symbol, respectively.

Let $L$ be a diagonally cyclic latin square in standard order.
 For any $c$ coprime to $n$, we define $L\times c$ to be the diagonally cyclic latin
square generated by operation~$\circ_c$, where $i\circ_c j =
(ic^{-1}\circ jc^{-1})c$\ (mod $n$), for each $i$, $j$ such that $0\leq i,j\leq
n-1$. (Informally, the first row of $L\times c$ is formed by
multiplying each column and symbol of the first row of $L$ by $c$.)
Again, such a transformation is invertible. For
the curious reader, other equivalences of diagonally cyclic latin
squares are given in Lemma~$2.1$ of \cite{gr1}.

So suppose that $L$ is a diagonally cyclic latin square of prime order $p$
containing the partial latin square $P$ defined as in the abstract.
Consider the isomorphic latin square
\[
L'=(((L\oplus_1(-s))\oplus_3 (-c))\times ((c'-c)^{-1}).
\]
Observe that $L'$ is in standard order, with symbol $(s'-s)/(c'-c)$ in row $0$
and column $1$.

In Section~\ref{sec7}, when we prove the main
result,  we thus assume, without any loss of generality, that
 $0\circ 0=0$ and $0\circ 1=j$.

In any diagonally cyclic latin square $L$, the set of symbols
$L(\alpha)=\{(i,\alpha+i,0\circ\alpha+i )\mid 0\leq i\leq n-1\}$
contains each row, each column and each symbol exactly once and is
thus a {\em transversal}. We denote $L(\alpha)$ as a {\em cyclic
transversal} of $L$. In fact, the cyclic transversals $L(\alpha)$,
$0\leq \alpha\leq n-1$ are pairwise disjoint. Thus any diagonally
cyclic latin square is orthogonal to the latin square defined by the
relation $i\circ j=j-i$.

The research in this paper is motivated by the following problem
posed by Alspach and Heinrich~\cite{alhe}: For each $k$, does there exist an
$N(k)$ such that if $k$ transversals of a partial latin square of
order $n>N(k)$ are prescribed, the square can always be completed?
The existence of idempotent latin squares for every order $n\neq 2$
shows that $N(1)=3$. H\"{a}ggkvist and Daykin~\cite{hada} have shown
that every partial $n\times n$ latin square where each row, column
and symbol is used at most $2^{-9}\sqrt{n}$ times is completable
whenever $n$ is divisible by $16$. However it is as yet unconfirmed
that $N(2)$ exists.
 Gr\"uttm\"uller~\cite{gr2} showed that
$N(k)\geq 4k-1$. Computational results on latin squares of small
order support the conjecture that $N(k)=4k-1$. 

As the above problem is difficult, it seems sensible to consider the
following modified version, proposed by Gr\"uttm\"uller~\cite{gr1}:
Does there exist an odd constant $C(k)$ such that if $k$ {\em
cyclically generated} diagonals of a partial latin square of odd
order $n\geq C(k)$ are prescribed, the square can always be
completed? Gr\"uttm\"uller showed that $C(2)=3$ (\cite{gr2}) and
that $C(k)\geq 3k-1$ for $k\geq 3$ (\cite{gr1}).

In this paper we provide some evidence that $C(3)$ may be equal to $9$.
Specifically, we show (in Theorem~\ref{main2}) that if $P$ is a partial
latin square of prime order $p>7$
of the form
\[
P=\{(i,c+i,s+i),(i,c'+i,s'+i),(i,c''+i,s''+i)\mid 0\leq i< p\}
\]
for some distinct $c$, $c'$, $c''$ and some distinct $s$, $s'$,
$s''$, then $P$ has a completion to a (diagonally cyclic) latin
square.

In our proof, in Section 2 we first identify a method to reorder
certain cyclically generated transversals within $B_{p,j}$ to form a
new diagonally cyclic latin square. This reordering or {\em trade}
is algebraically defined and the transversals are based on some
linear transformation of the quadratic residues mod $p$. We are thus
able to redefine the problem above in terms of simultaneous
equations, where instead of a precise solution to each equation, we
instead require some information about which coset the solution
belongs to. Here cosets are taken from the multiplicative group of
the field of prime order $p$.

When $p$ is large enough (specifically, when $p\geq 191$) in Section 4
we conveniently exploit Weil's theorem to obtain the required
solution. Computational methods are then employed in
Section~\ref{sec8} for the remaining, smaller primes.

In the next section we describe some combinatorial objects which are
equivalent to diagonally cyclic latin squares.

\section{Combinatorial equivalences}

Diagonally cyclic latin squares have a number of intriguing
equivalences. Firstly, diagonally cyclic latin squares of order $n$
are equivalent to transversals in the latin square $B_n$, which is
precisely the operation table for the integers modulo $n$. To see
this, let $L$ be a diagonally cyclic latin square with operation~$\circ$. 
Let $P\subseteq B_n$ be the partial latin square defined by
$P=\{(0\circ i-i,i,0\circ i)\mid 0\leq i\leq n-1\}$. Then $P$ is a
transversal. Conversely, let $P$ be a transversal of $B_n$. Then,
for each triple $(r,c,r+c \pmod{n})\in P$, define $0\circ c = r+c$.
Then $\circ$ generates a valid diagonally cyclic latin square.

Another equivalence involves placing $n$ semi-queens on a toroidal
$n \times n$ chessboard such that no two semi-queens may attack each
other. A semi-queen attacks any piece in its row and column, but
only on the {\em ascending} diagonals; i.e. those which begin in the
lower left and finish in the upper right. On a toroidal chessboard
these diagonals ``wrap around'' in the obvious fashion. Thus if we
replace the chessboard with the operation table for the integers
modulo $n$, we see that each queen must lie on a different symbol,
and no queens may share a common row or a common column. So a set of
$n$ semi-queens which may not attack each other on the next turn is
equivalent to a transversal within $B_n$. For more detail on the
semi-queen problem, we refer the reader to~\cite{vardi}.

Diagonally cyclic latin squares in standard order are also
equivalent to {\em complete mappings} of the cyclic group. We refer
the reader to~\cite{hsc} for more detail on complete mappings.

\section{How to reorder cyclic transversals within $B_{n,j}$ via number theory}\label{sec6}

Consider the diagonally cyclic latin square $B_{11,6}$ below. If we
replace each symbol in bold with its subscript, we obtain another
diagonally cyclic latin square. In fact, the bold symbols in the
first row are precisely the quadratic residues modulo $11$, and the
subscripts in the first row are obtained by multiplication by $4$.
This example is part of a general construction given in the
following theorem.
\begin{example}
\[
\begin{array}{|c|c|c|c|c|c|c|c|c|c|c|}
\hline
0 & 6 & {\bf 1_4} & 7 & 2 & 8 & {\bf 3_1} & {\bf 9_3} & {\bf 4_5} & 10 & {\bf 5_9}   \\
\hline
{\bf 6_{10}} & 1 & 7 & {\bf 2_5} & 8 & 3 & 9 & {\bf 4_2} & {\bf 10_4} & {\bf 5_6} & 0   \\
\hline
1 & {\bf 7_{0}} & 2 & 8 & {\bf 3_6} & 9 & 4 & 10 & {\bf 5_3} & {\bf 0_5} & {\bf 6_7}   \\
\hline
{\bf 7_8} & 2 & {\bf 8_{1}} & 3 & 9 & {\bf 4_7} & 10 & 5 & 0 & {\bf 6_4} & {\bf 1_6}    \\
\hline
{\bf 2_7} & {\bf 8_9} & 3 & {\bf 9_{2}} & 4 & 10 & {\bf 5_8} & 0 & 6 & 1 & {\bf 7_5}   \\
\hline
{\bf 8_6} & {\bf 3_8} & {\bf 9_{10}} & 4 & {\bf 10_{3}} & 5 & 0 & {\bf 6_9} & 1 & 7 & 2   \\
\hline
3 & {\bf 9_7} & {\bf 4_9} & {\bf 10_{0}} & 5 & {\bf 0_{4}} & 6 & 1 & {\bf 7_{10}} & 2 & 8   \\
\hline
9 & 4 & {\bf 10_8} & {\bf 5_{10}} & {\bf 0_{1}} & 6 & {\bf 1_{5}} & 7 & 2 & {\bf 8_{0}} & 3  \\
\hline
4 & 10 & 5 & {\bf 0_9} & {\bf 6_{0}} & {\bf 1_{2}} & 7 & {\bf 2_{6}} & 8 & 3 & {\bf 9_{1}}  \\
\hline
{\bf 10_{2}} & 5 & 0 & 6 & {\bf 1_{10}} & {\bf 7_{1}} & {\bf 2_{3}} & 8 & {\bf 3_{7}} & 9 & 4   \\
\hline
5 & {\bf 0_{3}} & 6 & 1 & 7 & {\bf 2_{0}} & {\bf 8_{2}} & {\bf 3_{4}} & 9 & {\bf 4_{8}} & 10  \\
\hline
\end{array}
\]
\label{exone}
\end{example}

\begin{theorem}
Let $(a,b,c)$ be a solution to the equation:
\begin{eqnarray}
(1-j)a^m+jb^m & \equiv & c^m \pmod{p} \label{egsactly}
\end{eqnarray}
where $p$ is a prime, $j\neq 1$, $j$ is coprime to $p$, $m\geq 2$,
$m$ divides $p-1$ and $a^m$, $b^m$ and $c^m$ are non-zero and pairwise distinct (mod~$p$).
Let $\alpha\neq 0$ and $\gamma$ be constants modulo $p$. Next,
replace each symbol in the first row of $B_{p,j}$ of the form
\[
\alpha\beta^m+\gamma
\]
(where $\beta\neq 0$) with the symbol
\[
\alpha(b/c)^m\beta^m+\gamma.
\]
Then this new first row generates a valid diagonally cyclic latin
square. \label{mth}
\end{theorem}

\begin{proof}
Let $L$ be the square array (at this stage we haven't proved that it
is latin) generated by the new first row.  
Since $b^m\neq c^m$, 
$(b/c)^m$ acts as a derangement on the set of $m$th powers, so the symbols in the first row of $L$ (and
thus in all rows) are distinct.

It suffices, then, to check that all symbols in the first column are
distinct. Within $B_{p,j}$, symbol $\alpha\beta^m+\gamma$ occurs in
column $j^{-1}(\alpha\beta^m+\gamma)$ in row $0$. So this cyclic
transversal contains the symbol
$(\alpha \beta^m + \gamma) (1-j^{-1})$ 
in column $0$. Within $L$ this is replaced by
\begin{align*}
(\alpha (b/c)^m\beta^m + \gamma) - j^{-1}(\alpha\beta^m+\gamma)
&= \alpha\beta^m((b/c)^m-j^{-1}) + \gamma (1-j^{-1})  \\
&= (\alpha\beta^m(a/c)^m+ \gamma) (1-j^{-1}).
\end{align*}
Again, as $(a/c)^m$ permutes the $m$th powers, the symbols in the
first column (and thus in all columns) are distinct.
\end{proof}

Observe that $7^2+1^2=2\times 5^2$. Thus, for $m=2$, the triple
$(a,b,c)=(7,1,5)$ is a valid solution of Equation~\eqref{egsactly}
for $j=(p+1)/2$ and any prime $p> 7$. The earlier
Example~\ref{exone} demonstrates this for $p=11$, $\alpha=1$ and
$\gamma=0$.

\section{Completing sets of cyclic transversals}\label{sec7}

We now focus on the problem of completing three arbitrary cyclic
transversals to a (diagonally cyclic) latin square. In this section,
we restrict ourselves to the case where $p$ is a prime. We denote the
finite field of size $p$ by $\GF(p)$.

We may assume, without loss of generality, that we have transversals
generated by $0$ and $j$ in cells $(0,0)$ and $(0,1)$, respectively.
We
must have $j\neq 1$ for this to be a valid partial latin square. Our
third transversal is arbitrary; assume that it is generated by
symbol $e$ in column $c$. So that no symbols repeat in a row or in
a column, we have as necessary conditions $c\not\in\{0,1\}$,
$e\not\in\{0,j,c\}$ and $e-c\neq j-1$. If $e=jc$, then
these three transversals are contained within $B_{p,j}$ and thus
complete to a latin square. So we henceforth assume also that
$e\neq jc$.

Our aim is to apply Theorem~\ref{mth} to $B_{p,j}$ to obtain a latin
square which contains the above three cyclic transversals. To
transform the third transversal, it is sufficient to find $m$,
$\alpha$,
$\beta$ and $\gamma$, with 
$m \mid p-1$, $\alpha\neq 0$, $\beta\neq 0$, such that
\begin{eqnarray}
jc & = & \alpha\beta^m +\gamma \label{eqn2} \\
e & = & \alpha x\beta^m +\gamma \label{eqn2a}
\end{eqnarray}
where $x=(b/c)^m$ for some appropriate solution $(a,b,c)$ to
Equation \ref{egsactly}.
Equivalently (considering the 
conditions of Theorem~\ref{mth}), we assume that $x$ is an $m$th
power modulo $p$, $x\not\in\{0,1\}$ and that if we define
\[
F(x)=\frac{1-jx}{1-j},
\]
then $F(x)$ is also an $m$th power and $F(x) \notin \{0,1\}$.

Equations~\eqref{eqn2} and \eqref{eqn2a} imply that a given $x$
determines $\gamma$ and the product~$\alpha\beta^m$:
\begin{eqnarray}
\gamma & = & \frac{xjc-e}{x-1}, \\
\alpha \beta^m & = & \frac{e-jc}{x-1}. \label{eqn3}
\end{eqnarray}

Since $x\neq 1$ these equations are well-defined. Given a valid $x$ and
$m$, it is always possible to choose such $\alpha$, $\beta$ and $\gamma$.
However, in this process of transformation we do not wish to alter
the first two transversals in columns~$0$ and $1$.

Equivalently, $-\gamma/\alpha$ and $(j-\gamma)/\alpha$ must not be
$m$th powers modulo $p$ (although they may be equal to $0$). But
Equation~\eqref{eqn3} implies that $\alpha$ is a non-zero $m$th power
if and only if $(e-jc)/(x-1)$ is a non-zero $m$th power. Thus, the
following expressions must not be non-zero $m$th powers modulo $p$:
\begin{align*}
G(x) &= \frac{-\gamma(x-1)}{e-jc} = \frac{e-xjc}{e-jc} \\
H(x) &= \frac{(j-\gamma)(x-1)}{e-jc} = \frac{xj(1-c)+e-j}{e-jc} \label{eqn6}
\end{align*}
Note that if $H(x)\neq 1$ then $x\neq 1$. In turn, if $x\neq 1$ then
$F(x)\neq 1$. So if we assume that $H(x)\neq 1$ it is unnecessary to
specify $x\neq 1$ and $F(x)\neq 1$ as conditions. In summary, we
have the following.

\begin{lemma}
Let $p$ be prime. Let $c$, $e$ and $j$ be residues mod $p$ such
that $c\not\in\{0,1\}$, $j\not\in\{0,1\}$ and
$e\not\in\{0,\, j,\, c,\, c+j-1,\, jc\}$. Let $x\neq 0$ be an $m$th power mod
$p$ such that $F(x)$ is a non-zero $m$th power mod $p$,
 and $G(x)$ and $H(x)$ are {\em not}
non-zero $m$th powers mod $p$. Then, the three cyclic transversals
generated by $0\circ 0=0$, $0\circ 1=j$ and $0\circ c=e$
 can be completed to a diagonally cyclic latin square.
\label{pair}
\end{lemma}

If we fix a value of $x$ there will be some inappropriate choices of
$c$, $e$ and $j$. However, we next show that for sufficiently large
primes, it is possible to find such a $x$ for any choice of $c$, $e$
and $j$ that satisfies the necessary conditions outlined above.

The techniques that follow are similar to those applied to constructing cyclic
triplewhist tournaments in~\cite{bur}.
We need the following version of Weil's Theorem (see~\cite{lidl}, Theorem~5.38).
Recall that a {\em multiplicative character} $\chi$ over $\GF(p)$
is a homomorphism from the multiplicative group of $\GF(p)-\{0\}$ into the multiplicative group of unitary complex numbers.
When Weil's theorem is applied it is also understood that $\chi(0)=0$ for any multiplicative character $\chi$.

\begin{theorem} (Weil's Theorem)
Let $\chi$ be a multiplicative character of the finite field
$\GF(p)$ with order $m>1$ and let $f$ be a polynomial of $\GF(p)[x]$ of degree $d$
which is not of the form $kg^m$ for some constant $k\in \GF(p)$ and some
$g\in \GF(p)[x]$. Then:
\[
\left|\sum_{c\in \GF(p)} \chi(f(c))\right|\leq (d-1)p^{1/2}.
\]
\label{veil}
\end{theorem}

For $m=2$, we will use the {\em quadratic character} $\eta$
defined as follows on $\GF(p) -\{0\}$:
\[
\eta(x) =
    \begin{cases}
        1 & \textnormal{ if $x$ is a quadratic residue;} \\
        -1 & \textnormal{ otherwise.}
    \end{cases}
\]
For $m=2$ we can obtain a more explicit version of Weil's Theorem (\cite{lidl}, Theorem~5.48):
\begin{theorem}
Let $\eta$ be the quadratic character of $\GF(p)$ and let
$f=a_2x^2+a_1x+a_0\in \GF(p)[x]$ with $p$ odd and $a_2\neq 0$. If $a_1^2-4a_0a_2\neq 0$ then:
\[
\sum_{c\in \GF(p)} \eta(f(c))=-\eta(a_2).
\]
\label{quad}
\end{theorem}

We can now show the main theorem of this section:
\begin{theorem}\label{themain}
If $P$ is a partial latin square of prime order $p\geq 191$
comprising of three arbitrary transversals, then $P$ has a
completion to a (diagonally cyclic) latin square.
\end{theorem}

\begin{proof}
Setting $m=2$, it suffices to show that the conditions of
Lemma~\ref{pair} hold for any valid choices of $j$, $c$, and $e$. We
wish to show that there exists some $x\in \GF(p)$ such that
$\eta(x)=1$, $\eta(F(x))=1$, ($\eta(G(x))=-1$ or $0$) and
($\eta(H(x))=-1$ or $0$). In fact, we show that there exists
specifically such an $x$ for which $\eta(x)=1$, $\eta(F(x))=1$,
$\eta(G(x))=-1$ and $\eta(H(x))=-1$.

Equivalently, we will show that for $p\geq 191$, the following set $A$
is non-empty:
\[
A = \{x\mid \eta(x)= \eta(F(x))=1 \textnormal{ and } \eta(G(x))=
\eta(H(x))=-1\}.
\]
Define
\[
J(x)= [1+\eta(x)][1+\eta(F(x))][1-\eta(G(x))][1-\eta(H(x))]
\]
and
\[
S=\sum_{x\in \GF(p)} J(x).
\]
If any of $x$, $F(x)$, $G(x)$ or $H(x)$ equals $0$, then $J(x) \leq 8$.
If 
$x\in A$ then $J(x) = 16$.  For other values 
of $x \in \GF(p)$, $J(x)=0$.  Thus $S \leq 16 \left| A
\right|+32$. So it suffices to show that $S > 32$.

By expanding $S$ and using the fact that $\eta$ is a homomorphism,
we can express $S$ as $p$ plus a series of terms of the form
$\sum_{c\in \GF(p)} \eta(K(x))$, where $K(x)$ is a
product of non-repeated factors from the set $\{x,F(x),-G(x),-H(x)\}$.

For any (non-constant) linear function $g(x)$, it is well-known that
\[
\sum_{x\in \GF(p)} g(x)=0.
\]
From the necessary conditions on $j$, $c$ and $e$, each of the
quadratics $xF(x)$, $xG(x)$, $xH(x)$ $F(x)G(x)$, $F(x)H(x)$ and
$G(x)H(x)$ have no repeated linear factors and thus each has a
non-zero discriminant.

It follows that the conditions of Weil's theorem hold for each of
the terms of the form $K(x)$ in the expansion of $S$. There are $4
\choose 2$ quadratic terms, so from Theorem \ref{quad}, these
contribute at least $-6$ to $S$. Next, there are ${4 \choose 3}$
cubic terms, so from Theorem \ref{veil} with $d=3$, these contribute
at least $-8\sqrt{p}$ to $S$. Finally there is one quartic term, so
from Theorem \ref{veil} with $d=4$, this contributes at least
$-3\sqrt{p}$ to $S$. Thus, $S \geq p - 11\sqrt{p}-6$. But
$p-11\sqrt{p}-6 > 32$ for any prime $p\geq 191$.  
\end{proof}

\section{Computational results}\label{sec8}

By using computational methods for primes less than $191$,
Theorem~\ref{themain} from the previous section can be improved to the following:
\begin{theorem}
If $P$ is a partial latin square of prime order $p>7$
of the form
\[
P=\{(i,c+i,s+i),(i,c'+i,s'+i),(i,c''+i,s''+i)\mid 0\leq i < p\}
\]
for some distinct $c,c',c''$ and some distinct $s,s',s''$,
then $P$ has a completion to a (diagonally cyclic) latin square.
\label{main2}
\end{theorem}
Note that in~\cite{gr1} it is shown that three cyclic transversals do not always
complete to diagonally cyclic latin squares of order $7$.

Two approaches were used to do the small cases computationally. Firstly, we
checked all the instances where the conditions of Lemma~\ref{pair} were satisfied
for an appropriate $m$.
This approach provided complete solutions
for all primes between $61$ and $181$ (inclusive),
and most solutions for primes between $11$ and $59$,
leaving a total of 
2076 exceptions (choices of $j$, $c$, and $e$) for smaller primes:
\[
\begin{array}{|l|c|c|c|c|c|c|c|c|c|c|c|c|c|}
\hline
\textnormal{ prime } & 
11 &
13 &
17 &
19 &
23 &
29 &
31 &
41 &
47 &
59 \\ 
\hline
\textnormal{ \#excep } & 
180 &
54 &
306 &
84 &
672 &
252 &
66 &
12 &
300 &
150 \\ 
\hline
\end{array}
\]
The exceptions $11 \leq p \leq 59$ were solved using a randomised depth
first search algorithm~\cite{gomes98boosting}.

\clearpage
\section{Case $p=11$}

Here we present the computational data for $p=11$. Each of the
following four columns contain the exceptions (triples $c$, $e$,
$j$) and the first row of the diagonally cyclic latin square which
was found by backtrack search for the given triple. The symbol `a'
denotes $10$.

{ \tiny
\begin{multicols}{4}
\noindent 
2 1 4 04175a29368 \\
2 1 8 08146a92573 \\
2 3 5 0538a429617 \\
2 3 9 09358a42716 \\
2 5 2 0258a149763 \\
2 5 9 09582176a43 \\
2 6 7 076194835a2 \\
2 6 8 08625391a47 \\
2 7 4 0472a185963 \\
2 7 5 057a6432918 \\
2 8 3 038279415a6 \\
2 8 10 0a846953217 \\
2 10 3 03a49852176 \\
2 10 7 07a49853621 \\
3 1 5 05a13879246 \\
3 1 7 07518624a39 \\
3 1 8 085169327a4 \\
3 1 9 09417a58326 \\
3 2 3 035291a8647 \\
3 2 4 049287316a5 \\
3 2 5 0582614a937 \\
3 2 7 07925a84613 \\
3 4 3 03148a92657 \\
3 4 7 07542a39186 \\
3 4 9 0964371a258 \\
3 5 2 02651a93748 \\
3 5 4 04a59372186 \\
3 5 8 08a57642139 \\
3 6 3 03962a58147 \\
3 8 9 091857a3642 \\
3 9 4 04791625a83 \\
3 9 8 083962a5714 \\
3 9 10 0a495831726 \\
3 10 3 038a1695724 \\
3 10 5 054a9276318 \\
3 10 9 094a7652138 \\
4 1 2 0285194a736 \\
4 1 7 0795164a283 \\
4 1 8 08721695a43 \\
4 2 3 037a2498165 \\
4 2 9 097524a8316 \\
4 3 4 04693a28571 \\
4 3 5 057938246a1 \\
4 3 10 0a483692157 \\
4 5 6 06925814a73 \\
4 5 7 07a65429138 \\
4 5 8 0869543a271 \\
4 6 4 04a86175329 \\
4 6 8 08376a42519 \\
4 6 9 09376a15482 \\
4 7 3 038271a5964 \\
4 7 5 05397421a68 \\
4 8 7 07328a94651 \\
4 8 9 0942817a635 \\
4 9 3 03a49826175 \\
4 9 4 04a59172683 \\
4 9 8 08379415a62 \\
4 10 4 0431a986257 \\
4 10 5 0538a296471 \\
4 10 6 0694a285713 \\
5 1 2 0248a196573 \\
5 1 3 038651a4279 \\
5 1 4 04a72186935 \\
5 1 5 0542a198637 \\
5 1 7 074621a8539 \\
5 2 3 0368521a479 \\
5 2 8 0865321a974 \\
5 3 4 04981376a25 \\
5 3 9 097263a8415 \\
5 4 5 058a1479632 \\
5 4 7 07a59428613 \\
5 4 8 08a65419273 \\
5 4 9 0937641a258 \\
5 4 10 0a695483217 \\
5 7 4 04a85713629 \\
5 7 5 0532971a468 \\
5 7 9 0916a748253 \\
5 8 5 0541a873269 \\
5 8 7 07915836a24 \\
5 9 3 0318297a465 \\
5 9 7 073a6941582 \\
5 9 8 08a53971642 \\
6 1 5 05742813a69 \\
6 1 8 08a75419632 \\
6 1 9 096574132a8 \\
6 2 6 0685142a973 \\
6 2 7 07358a24619 \\
6 4 5 05793148a62 \\
6 4 6 068572431a9 \\
6 5 3 0398a654172 \\
6 5 4 04a92758136 \\
6 5 7 07361a59428 \\
6 7 8 08591376a24 \\
6 8 6 063a2984715 \\
6 8 9 09615a83742 \\
6 9 3 03148a92657 \\
6 9 6 061a2798543 \\
6 10 4 048136a9257 \\
7 1 2 0296a481573 \\
7 1 3 038a5491627 \\
7 1 4 04928371a65 \\
7 1 5 054928317a6 \\
7 2 3 03a68152974 \\
7 2 9 095861427a3 \\
7 3 4 04172a83965 \\
7 3 7 07158a43962 \\
7 4 5 05397a24681 \\
7 4 8 085936a4271 \\
7 5 3 03a87625149 \\
7 5 9 09873625a14 \\
7 6 7 07a49386152 \\
7 6 8 0854a936271 \\
7 6 9 095187263a4 \\
7 6 10 0a398726514 \\
7 8 4 0472a938651 \\
7 8 5 054a7218639 \\
7 9 5 058174392a6 \\
7 9 7 07a83649152 \\
7 10 7 0748532a169 \\
7 10 8 0817624a935 \\
8 1 2 02a97386154 \\
8 1 3 035a9642187 \\
8 1 4 04392a86157 \\
8 1 6 068a5439172 \\
8 1 9 098627531a4 \\
8 2 5 0596a438271 \\
8 2 7 074a8635219 \\
8 2 8 08a63975241 \\
8 3 8 081762953a4 \\
8 3 9 097624a8351 \\
8 5 3 0317962a548 \\
8 5 4 047139825a6 \\
8 6 4 04953a18627 \\
8 6 5 05493a28617 \\
8 6 7 07a43981652 \\
8 7 3 03a69148725 \\
8 7 6 065a2938741 \\
8 7 8 084692157a3 \\
8 7 9 0938614a725 \\
8 7 10 0a685139742 \\
8 9 4 048a2751963 \\
8 9 5 05417a32986 \\
8 10 7 07185394a26 \\
8 10 8 08721695a43 \\
9 1 3 036a9248715 \\
9 1 9 09642a83715 \\
9 1 10 0a875439216 \\
9 2 4 04a95381726 \\
9 3 5 05192874a36 \\
9 3 7 074a1958632 \\
9 3 8 084916a5732 \\
9 4 3 0397165a248 \\
9 4 5 05a63721948 \\
9 5 7 0741863a259 \\
9 5 9 0917268a354 \\
9 6 4 049853a2761 \\
9 6 5 0538a149762 \\
9 6 7 075a2481963 \\
9 7 8 081596a4372 \\
9 8 2 02a57941386 \\
9 8 3 037916a5482 \\
9 8 9 0957a643281 \\
9 10 3 036971542a8 \\
9 10 4 048197265a3 \\
9 10 8 081762953a4 \\
9 10 9 095721863a4 \\
10 1 7 076a5394281 \\
10 2 5 057913864a2 \\
10 2 8 08653a14972 \\
10 2 10 0a685139742 \\
10 3 4 0472a185963 \\
10 4 3 035a1972684 \\
10 4 9 0965217a384 \\
10 6 3 035192a8746 \\
10 6 9 0935a841726 \\
10 7 8 08653a42917 \\
10 8 2 02579136a48 \\
10 8 4 0495a631728 \\
10 8 7 076514932a8 \\
10 9 5 0536a184279
\end{multicols}
}

\section{Acknowledgements}

The authors wish to thank Dr.\ Julian Abel for directing them to~\cite{bur}.

\bibliographystyle{plain}

\begin{thebibliography}{1}

\bibitem{alhe}
Brian Alspach and Katherine Heinrich.
\newblock Matching designs.
\newblock {\em Australas. J. Combin.}, 2:39--55, 1990.
\newblock Combinatorial mathematics and combinatorial computing, Vol.\ 2
  (Brisbane, 1989).

\bibitem{bur}
Marco Buratti.
\newblock Existence of {$Z$}-cyclic triplewhist tournaments for a prime number
  of players.
\newblock {\em J. Combin. Theory Ser. A}, 90(2):315--325, 2000.

\bibitem{hada}
David~E. Daykin and Roland H{\"a}ggkvist.
\newblock Completion of sparse partial {L}atin squares.
\newblock In {\em Graph theory and combinatorics (Cambridge, 1983)}, pages
  127--132. Academic Press, London, 1984.

\bibitem{gomes98boosting}
Carla~P. Gomes, Bart Selman, and Henry Kautz.
\newblock Boosting combinatorial search through randomization.
\newblock In {\em Proceedings of the Fifteenth National Conference on
  Artificial Intelligence ({AAAI}'98)}, pages 431--437, Madison, Wisconsin,
  1998.

\bibitem{gr1}
Martin Gr{\"u}ttm{\"u}ller.
\newblock Completing partial {L}atin squares with two cyclically generated
  prescribed diagonals.
\newblock {\em J. Combin. Theory Ser. A}, 103(2):349--362, 2003.

\bibitem{gr2}
Martin Gr{\"u}ttm{\"u}ller.
\newblock Completing partial {L}atin squares with prescribed diagonals.
\newblock {\em Discrete Appl. Math.}, 138(1-2):89--97, 2004.
\newblock Optimal discrete structures and algorithms (ODSA 2000).

\bibitem{hsc}
Hsiang J, Shieh Y, and Chen Y.
\newblock The cyclic complete mappings counting problems.
\newblock In {\em PaPS: Problems and Problem Sets for ATP Workshop in
  conjunction with CADE-18 and FLoC 2002}. Copenhagen, Denmark, 2002.

\bibitem{lidl}
Rudolf Lidl and Harald Niederreiter.
\newblock {\em Finite fields}, volume~20 of {\em Encyclopedia of Mathematics
  and its Applications}.
\newblock Addison-Wesley Publishing Company Advanced Book Program, Reading, MA,
  1983.
\newblock With a foreword by P. M. Cohn.

\bibitem{vardi}
Igor Rivin, Ilan Vardi, and Paul Zimmermann.
\newblock The {$n$}-queens problem.
\newblock {\em Amer. Math. Monthly}, 101(7):629--639, 1994.

\end{thebibliography}

\end{document}